\documentclass[11pt]{article}
\usepackage{amssymb}
\usepackage{amsthm}
\usepackage{amsmath,amsfonts,amssymb}

\usepackage[mathscr]{eucal}
\usepackage{exscale}
\usepackage{natbib}
\usepackage{bm}
\usepackage{eqlist}
\usepackage[dvipsnames]{color}
\usepackage[Lenny]{fncychap}
\usepackage{multirow}
\usepackage{graphicx}
\usepackage{algpseudocode}
\usepackage{algorithm}
\usepackage{caption}
\usepackage{hyperref}
\usepackage{framed}
\usepackage{color}
\usepackage{booktabs}
\usepackage{longtable}

\usepackage{framed}
\usepackage{algpseudocode}

\usepackage{tabu}
\usepackage{tikz}
\usepackage{ulem}
\usepackage{soul}

\usepackage[utf8]{inputenc} 
\usepackage[english]{babel} 
\usepackage[T1]{fontenc}    
\usepackage{amssymb,amsmath,amsfonts} 
\usepackage{braket} 
\usepackage{caption}
\usepackage{algorithm}

\hypersetup{
    colorlinks=true,
    linkcolor=blue,
    citecolor=blue,
    citebordercolor={1 1 0},
}

\newtheorem{theorem}{Theorem}[section]
\newtheorem{lemma}[theorem]{Lemma}
\newtheorem{proposition}[theorem]{Proposition}

\newtheorem{definition}[theorem]{Definition}
\newtheorem{remark}[theorem]{Remark}

\font\bigbf=cmbx10 scaled \magstep3

\begin{document}

\title{\bigbf  Internalizing sensing externality via matching and pricing for drive-by sensing taxi fleets}

\author{Binzhou Yang$^a$
\quad Ke Han$^{b,}\thanks{Corresponding author, e-mail: kehan@swjtu.edu.cn;}$
\quad Shenglin Liu$^a$
\quad Ruijie Li$^a$
\\\\
 $^a$\textit{\small School of Transportation and Logistics, Southwest Jiaotong University}
 \\
 $^b$\textit{\small School of Economics and Management, Southwest Jiaotong University}
}

\maketitle

\begin{abstract}
   Drive-by sensing is a promising data collection paradigm that leverages the mobilities of vehicles to survey urban environments at low costs, contributing to the positive externality of urban transport activities. Focusing on e-hailing services, this paper explores the sensing potential of taxi fleets, by designing a joint matching and pricing scheme based on a double auction process. The matching module maximizes the sensing utility by prioritizing trips with high sensing potentials, and the pricing module allocates the corresponding social welfare according to the participants' contributions to the sensing utility. We show that the proposed scheme is allocative efficient, individually rational, budget balancing, envy-free, and group incentive compatible. The last notion guarantees that the participants, as a cohort, will end up with the same total utility regardless of mis-reporting on part of its members. Extensive numerical tests based on a real-world scenario reveal that the sensing externality can be well aligned with the level of service and budget balance. Various managerial insights regarding the applicability and efficacy of the proposed scheme are generated through scenario-based sensitivity analyses.  
\end{abstract}


\section{Introduction}\label{secIntro}

The uprise of cheap and mobile sensors has offered unprecedented opportunity to scan the physical environment of a large urban area at low costs. Drive-by sensing (DS), where the sensors are mounted onto moving vehicles (e.g. taxis, buses), is one type of such ubiquitous sensing that has received increased popularity due to the high mobility of the hosts and low maintenance costs. DS has been applied to various sensing scenarios such as air quality \citep{Gao2016, SHS2021}, traffic state \citep{YFLRL2021, GQDRJ2022}, noise \citep{Cruz2020a, Cruz2020b}, heat island \citep{TRMSO2020}, parking availability \citep{BAD2017, MJKCXGT2010}, and built environment  \citep{ PDB2012, AD2017,XCPJZN2019}.

Data, information, and knowledge generated from drive-by sensing constitute a positive externality of urban transport activities, to be explored in the era of smart cities. However, new challenges arise when one pursues such an externality while ensuring the efficiency and performance of relevant transport systems. In this work, we focus on e-hailing services with taxis as sensing hosts, as they have been widely consider for drive-by sensing due to their high spatial mobility, long operating hours and low maintenance costs \citep{BBLARB2016, BAD2017, ML2022, ZMLL2015}. However, the sensing efficacy of taxi fleets is limited \citep{OASSR2019, JHL2023}, primarily restricted by spatially heterogeneous trip demands and profit-oriented operations, which lead to unbalanced spatial distribution of the fleet (sensors), with undesirable consequences such as sensing blindspot and sampling bias.

To tackle such a challenge in taxi-based sensing, several interventional or incentivizing schemes are proposed to improve the distribution of sensing taxis \citep{ADFS2021, Masutani2015, GQ2022, FJLQG2021, C2020, XCPJZN2019}; see Section \ref{subsecLRt} for more details. Primarily focusing on taxi routing or drivers' subsidies, these mechanisms are shown to be effective to some extents, but are rendered impractical by either compromising the level of service (e.g. taking detoured routes) or requiring budgets that are difficult to secure in real-world operations (e.g. driver subsidies). In this work, we propose a matching and pricing scheme for e-hailing services, with explicit focus on the  sensing externality while ensuring level of service and budget balance. Specifically, the matching rule is based on sensing utility maximization, which selects trip requests with high sensing potentials. The pricing scheme allocates social welfare from matched trips according to the participants' contributions to the sensing externality. As the matching rule is strongly pitched towards riders, which could result in inefficiencies in the pick-up phase, a double auction process is employed to manage the level of service in case of long pick-up distances or wait times. The theoretical and practical contributions of this work are as follows:
\begin{itemize}
\item {\bf Conception and model:} This work is the first to explore the sensing externality of urban e-hailing services, by employing pricing to internalize such externality, and a double auction mechanism to manage the level of service. 

\item {\bf Theoretical properties:} The proposed matching and pricing scheme is theoretically proven to be allocative efficient (AE), individually rational (IR), and budget balancing (BB) and envy-free. We further show a property of {\it group incentive compatibility} (G-IC), which means the cohort (of drivers and riders) will always end up with the same total utility regardless of the mis-reporting behavior of its members.

\item {\bf Practical significance:} A simulation study based on real-world data demonstrates the effectiveness of the proposed scheme in achieving far greater sensing efficacy while ensuring the level of service and generating revenue for the platform. Experimental and comparative studies (with the classical VCG scheme) reveal managerial and practical insights of the scheme, in relation to demand distribution, fleet size, and mis-reporting behavior. 
\end{itemize}

The rest of this paper is organized as follows. Section \ref{secLR} provides an overview of relevant studies. Section \ref{secBP} offers some background material, including the new notion of opportunity cost. The DS matching and pricing scheme is elaborated in Sections \ref{secMP}. Section \ref{secSim} conducts numerical and simulation case studies, and Section \ref{secCon} offers some concluding remarks.

\section{Related work}\label{secLR}

This section offers a review of existing research in related areas, and highlights the unique features and contributions relative to these studies.

\subsection{Taxi drive-by sensing}\label{subsecLRt}

Conventionally, the sensing capabilities of taxi fleets are investigated in the context of opportunistic sensing, where data are collected passively from instrumented vehicles with minimum interference to the operation of the taxi fleet. \cite{BBLARB2016} use empirical taxi data in Rome to show that a fleet of 120 vehicles can achieve 80\% coverage of the downtown area within 24 hours. \cite{BAD2017} consider taxi fleet's suitability to scout on-street parking availability in San Francisco and found that about 500 taxis can adequately cover 90\% of road segments, although their distribution is spatially and temporally heterogenous.  \cite{ZMLL2015} assess the expected proportion of spatial grids in Beijing and Shanghai that can be covered in an hour, which is over 50\% with 1,700 and 1,900 taxis, respectively. However, the marginal gain of such coverage is limited due to the spatially unbalanced distribution of taxis. \cite{OASSR2019} use a ball-in-bin model to analyze the sensing capability of taxi fleets in several major cities, and found that taxi fleets have considerable yet limited sensing capabilities, primarily due to the heterogenous distribution of taxi trajectories.

To improve the sensing coverage of taxi fleets, various operational interventions or incentives are proposed in the literature, primarily from the perspective of routing between a given origin-destination pair. \cite{ADFS2021} propose an $\varepsilon$-perturbed route set based on the shortest route between a given origin-destination pair using the A-star algorithm, such that the spatial coverage of a taxi fleet can be optimized within such route detour set. \cite{Masutani2015} makes routing recommendations to relevant vehicles based on centralized decision making that seeks to maximize the sensing quality. \cite{GQ2022} focus on the cruising phase of taxi services by providing routing guidance to vacant taxis. On the incentivizing schemes, \cite{FJLQG2021} design a joint scheduling and pricing scheme based on one-sided auction, which rewards drivers for taking routes within an acceptable detour range from the shortest one, in the interest of covering more areas. \cite{C2020} and \cite{XCPJZN2019} form routing advice for vacant taxis with a limited incentive budget to maximize their sensing gain while cruising for the next customer.

Despite the widespread recognition of taxi fleets' sensing potentials, studies that aim to enhance their sensing capabilities have mainly resorted to route guidance for in-service or vacant taxis. This work is the first to explore the two-sided matching of drivers and riders, and use pricing as a means to internalize sensing externality. By coordinating demand and supply, the proposed scheme offers greater potential for sensing maximization than supply maneuvers considered in existing research.

\subsection{Matching and pricing in e-hailing services}

In this paper, we primarily focus on methods and mechanisms related to matching and pricing in e-hailing services. The reader is referred to \cite{AESW2012} and \cite{FDOBWK2013} for more general discussion of these topics.

The bipartite drive-rider matching problem have been studied from the perspectives of drivers, riders and the platform. Some key attributes or objectives in these problems include waiting time \citep{ASWFR2017, LNL2020}, detour/travel distance \citep{AESW2011, PXZLKA2015}, matching rate and demand-supply balance \citep{LS2015, SASG2015, QU2017, MJ2017, SASG2016, YWW2002}, profits and fairness \citep{BGTMMT, FDKODBCW2015, GMT2016}, and matching stability \citep{WAE2018}. 

The design of pricing policies has frequently resorted to auctions, where the platform (auctioneer) processes bidding information submitted by riders and drivers (participants) and offers a price for each match (transaction). An ideal auction outcome should satisfy individual rationality (IR), incentive compatibility (IC), allocative efficiency (AE) and budget balance (BB). As it is impossible to design a pricing scheme that satisfies AE, IR and IC without incurring a deficit (not BB) \citep{MS1983}, rendering popular pricing schemes such as VCG difficult to implement in practice, many studies attempt to trade AE or IC for BB. For example, \cite{ZZGSY2014} design a fixed pricing scheme and a two-sided reserve pricing scheme in a ride-sharing system to bound the deficit by sacrificing efficiency. \cite{ZWZ2016} propose a discounted trade reduction scheme for dynamic ride-sharing pricing, which achieve IR, IC, BB but not AE. \cite{ZWB2018} study an auction model for a one-sided ride-sharing market with variable reserve price constraints, which achieves IC, IR, and BB but at the expense of social benefit, which is bounded below by half of the optimal social benefit. \cite{LNL2020} consider both operational cost and schedule displacement in a carpool matching scenario, by developing a single-sided reward pricing policy as a robust and deficit-free alternative to VCG. However, underreporting cannot be ruled out (not IC). \cite{LNL2022} propose a trilateral matching problem for a quantity-based demand management system, which can eliminate any deficit arising from the VCG policy, and promotes ride-sharing at a relatively small cost to system's efficiency.

Unlike the aforementioned studies, this work focuses on aligning e-hailing operations with sensing externality. This is achieved by maximizing the spatial-temporal sensing gain in the matching stage and allocating social welfare according to their contributions to the sensing externality. As a result, the scheme satisfies AE, IC, BB, and group-IC, ensuring its operational feasibility in practice.

\section{Background and preliminaries}\label{secBP}
This section describes several elements that are essential to the main discussion of the matching and pricing scheme, namely the cost structure (Section \ref{subsecCS}-\ref{subsecOC}), the double auction mechanism (Section \ref{subsecDAP}), and the quantification of sensing externality (Section \ref{subsecDSquant}). We begin with mathematical symbols and notations to be used in this paper, shown in Table \ref{tabmath}.

	\setlength\LTleft{0pt}
	\setlength\LTright{0pt}
	\begin{longtable}{@{\extracolsep{\fill}}rl}
		\caption{Mathematical symbols}	\label{tabmath} \\
			\hline
			\multicolumn{2}{l}{Sets and indices}     \\
			\hline
			$\mathcal{D}$    & Set of drivers, indexed by $d\in \mathcal{D}$ 
			\\
			$\mathcal{R}$      & Set of riders, indexed by $r\in \mathcal{R}$
			\\
			$\mathcal{D}^*$ & Set of matched drivers
			\\
			$\mathcal{R}^*$ & Set of matched riders
			\\
			$T$ & Set of sensing intervals $t\in T$
			\\
			$\mathcal{E}_t$ & Set of decision epochs $\varepsilon_t^i$, $i=1,\ldots, n$ within $t$
			\\
						\hline
			\multicolumn{2}{l}{Parameters and constants}                                                                                                                              \\
			\hline
			$l=(d,r)$               & A potential match of driver $d$ and rider $r$
			\\
			$s_{r}$     & Origin of the trip request made by rider $r$
	          	\\
			 $t_{r}$     & Destination of the trip request made by rider $r$
			\\
			$\tau _{dr}$ & Pick-up distance from the current location of driver $d$ to rider $r$
			\\
			$\tau_{d}^{\text{min}}$       & Minimum pick-up distance for driver $d$ to the nearest rider 
			\\
			$\tau_{r}^{\text{min}}$       & Minimum pick-up distance for rider $r$ from the nearest driver 
			\\
			$b_d$  & Cost per unit pick-up distance for driver $d$
			\\
			$\delta_r$    & Compensation per unit  pick-up distance for rider $r$
			\\
			$P_d(b_d)$ & Driver's valuation of a matched trip
			\\
			$P_r(\delta_r)$ & Rider's valuation of a matched trip
			\\
			$h_{r}$       & Distance from rider $r$'s origin $s_r$ to destination $t_r$, following the shortest route
			\\
			$\alpha$      & Monetary value per unit travel distance for driver (public information)
			\\
			$\beta $     & Monetary value per unit travel distance for rider (public information) 
			\\
			$f(t_r)$   & Opportunity cost that depends on the destination $t_r$ of a trip 
			\\
			$\sigma_{dr}$   & Social welfare of a matched trip $l=(d, r)$ 
			\\
			$\zeta_{dr}$   & Sensing externality of a matched trip $l=(d, r)$ 
			\\                   
			\hline
			\multicolumn{2}{l}{Decision variables}                                                                                                                              \\
			\hline
			$x_{dr}$ & Equals $1$ if driver $d$ and rider $r$ are matched, and 0 otherwise        
			\\
			$q_d$ & Platform's payment to the driver  
			\\
			$q_r$ &   Platform's charge to the rider                                                                                                                                        \\\hline
	\end{longtable}

\subsection{Cost structure}\label{subsecCS}

An online e-hailing platform collects relevant trip information submitted by drivers and riders in a decision-making epoch (e.g. several minutes). Every rider $r$ expects to complete a planned trip with origin $s_r$ and destination $t_r$.

We consider a potential match $l=(d,r)$, where the driver $d$ and rider $r$ valuate the trip independently. In particular, the driver's valuation (desired payment) is expressed as:
\begin{equation}\label{Ud} 
P_d(b_d)=\alpha \cdot h_{r}+b_d\cdot(\tau_{dr}-\tau^{\text{min}}_d) + f(t_r)
\end{equation} 
\noindent where the first term on the right hand side is the cost for serving the requested trip; the second term is the additional cost for extra pick-up distance relative to the nearest rider. The third term $f(t_r)$, which is expressed in terms of the destination of the trip $t_r$, is introduced to reflect the opportunity cost, which will be elaborated in Section \ref{subsecOC}.

Regarding the rider $r$, the valuation (willingness to pay) is:
\begin{equation}\label{Ur}
P_r(\delta_r) =\beta \cdot h_r -\delta_r\cdot(\tau_{dr}-\tau^{\text{min}}_r )
\end{equation}

\noindent where the first term concerns with the total distance traveled, and we typically set $\beta>\alpha$; the second term is the compensation associated with extra pick-up distance relative to the nearest driver.

The formulae \eqref{Ud} and \eqref{Ur} simultaneously consider the travel distance of the requested trip, the driver’s pick-up phase and the rider’s waiting phase. The following are implied: 
\begin{itemize}
\item For longer travel distance, the driver is expected to receive more compensation, and the rider is expected to pay more.
\item A driver/rider not matched to her nearest counterpart is expected to receive compensation that is proportional to the extra pick-up time. In this case, the participant makes independent bids based on her personal preference (i.e. $b_d$ or $\delta_r$).

\item The driver's valuation of a trip also depends on her assessment of subsequent cost of finding the next order, a concept referred to as the {\it opportunity cost}. We argue that the opportunity cost may influence the driver's willingness to accept a remote trip (i.e. trips ending with low local trip demands).
\end{itemize}

\subsection{The opportunity cost}\label{subsecOC}

When a driver finishes an order, she needs to cruise for a certain period, the duration of which is dependent on the probability that a trip request is received. If the destination of the last order has a low trip count (e.g. rural areas), the driver incurs a non-negligible cost because she may have to travel a long distance in search of the next order. Such an {\it opportunity cost} may influence a driver's valuation of a trip request because of its destination $t_r$, and hence her willingness to accept it. 

To elaborate the opportunity cost, we begin with the notion of {\it order prospect}, denoted $p(t_r)$, which measures the likelihood of receiving a request from a given location $t_r$. To fix the idea, we mesh the study area into grids $g\in G$, and obtain the probability densities of trip requests $\big\{n_g\geq 0: g\in G,\,\sum_{g\in G}n_g=1\big\}$ (such a distribution can be estimated based on historical request data). The order prospect is defined as a weighted sum of the probability densities:
\begin{equation}\label{prodef} 
p(t_r)=\sum_{g\in G}w_{g, t_r}\cdot  n_g
\end{equation}
\noindent where the weight $w_{g,t_r}$ measures the relevance of grid $g\in G$ to $t_r$, and we stipulate that $w_{g,t_r}$ decays as $g$ is further away from $t_r$. Formula \eqref{prodef} conforms to the following intuitions.
\begin{itemize}
\item The order prospect $p(t_r)$ depends on the probability densities of trip requests in the destination $t_r$ as well as near-by areas.

\item Such a dependence is smaller for girds further away from the destination $t_r$.
\end{itemize}
 
The opportunity cost $f(p(t_r))$ is expressed as a monotonically decreasing function of the order prospect $p(t_r)$. Illustrative examples of order prospects and opportunity costs are provided in a real-world case study, in Section \ref{subsecillex}.

\subsection{The double auction process}\label{subsecDAP}

	In a double auction process, the platform assumes the role of the auctioneer, while the drivers and riders act as participants.

\begin{definition}(Double auction)
We consider a bilateral ride-sourcing market composed of riders (buyers), drivers (sellers) and the platform (auctioneer). The participants (riders \& drivers) independently perform valuation of a potential match, based on both public (e.g. O-D pair, rates per distance traveled) and private (rate per unit waiting time) information about the trip. It is possible that the participants misreport their private information for utility gain. After collecting the valuation information provided by all participants, the platform determines the trading scheme of the market (i.e. matched pairs and their prices).
\end{definition}

Following the classic auction theory \citep{MS1983}, we assume that there may be incentives for the participants to falsely report $b_{d}$ and $\delta_{r}$ in the bidding process. The reason is two-fold:
\begin{itemize}
    \item[(1)] In contrast to public information such as $\alpha$, $\beta$ and origin-destination pair of the trip, $b_d$ and $\delta_r$ are private information, and heterogeneous for different participants.
    \item[(2)] The parameters $b_{d}$ and $\delta_{r}$ directly affect the participants' expected pay-offs.
\end{itemize}

The platform uses \eqref{Ud} and \eqref{Ur} to process the information submitted by the participants, together with an estimate of the drive-by sensing rewards based on relevant trip information (such as the shortest path between an O-D pair), when making real-time matching and pricing decisions.

\subsection{Quantifying drive-by sensing (DS) externality}\label{subsecDSquant}

The drive-by sensing externality is quantified based on a meshed space and discrete time intervals. Let $G$ be a partition of the target area, and $g\in G$ represents a spatial grid (of size 1km$\times$1km in our paper). Let $t\in T$ be a sensing interval (e.g. 1 hour). Then, the sensing quality of grid $g$ during $t$ is expressed as follows \citep{HJNLL}:
$$
\phi_{g,t}(N_{g,t})=N_{g,t}^{\lambda} \qquad \lambda\in (0,1).
$$
\noindent Here, the sensing quality $\phi_{g,t}$ is a function of $N_{g,t}$, which is the number of distinct vehicles that have visited $g$ during $t$. We stipulate that $\phi_{g,t}(\cdot)$ is an increasing function, as more vehicle visits provide more information, but the marginal gain $\phi_{g,t}'(\cdot)$ is monotonically decreasing. The property of decreasing marginal sensing gain is important for optimizing the sensing quality over multiple grids, as it prevents over-concentration of vehicles in a few grids. Note that the value of $\lambda$ depends on the underlying application of drive-by sensing \citep{JHL2023}. For the case study in this paper, we consider air quality monitoring with $\lambda=0.2$ \citep{HJNLL}.

The overall spatial-temporal sensing utility afforded by a fleet of vehicles instrumented with mobile sensors in a target area can be quantified as:
\begin{equation}\label{eqnPhi}
\Phi=\sum_{t\in T}\mu_t \sum_{g\in G} w_g \phi_{g,t}(N_{g,t}),
\end{equation}
\noindent where $\mu_t$ and $w_g$ are temporal and spatial weights associated with drive-by sensing, such that $\sum_{t\in T}\mu_t=1$, $\sum_{g\in G}w_g=1$. These weights can be user-defined to reflect sensing priorities of certain location during certain time.

Within the context of online vehicle-passenger matching, the quantification of sensing externality is more subtle than \eqref{eqnPhi}. Each sensing period $t\in T$ is further divided into smaller decision-making epochs $\varepsilon_t^i$ (e.g. 3 min), $i=1,\ldots, n$, in which the platform collects relevant information on riders and drivers, and makes matching and pricing decisions accordingly. In this process, the drive-by sensing gain of a potential match $l=(d, r)$ considered in epoch $\varepsilon_t^i$ is calculated as
\begin{equation}\label{zetadr}
\zeta_{dr}=\sum_{g \in G_l} \big((N_{g,t}(\varepsilon_t^i)+1)^{\lambda}-N_{g,t}(\varepsilon_t^i)^{\lambda}\big),
\end{equation}
\noindent where $G_l\subset G$ is the subset of grids traversed by the matched trip $l$, $N_{g,t}(\varepsilon_t^i)$ denotes the number of vehicles that have covered (or are scheduled to cover) grid $g$ by the beginning of epoch $\varepsilon_t^i$. In prose, $\zeta_{dr}$ represents the marginal sensing externality contributed by the extra trip $l=(d,r)$.

\section{The joint matching-pricing scheme}\label{secMP}

This section elaborates the proposed matching and pricing scheme, along with its mathematical properties, in Sections \ref{subsecDSdesc}-\ref{subsecDSprop}. As a benchmark for comparison, we also describe the VCG scheme in Section \ref{subsecVCG}.


The entire analysis horizon is divided into decision epochs. In each epoch the platform receives relevant information of the participants, including the sets of unmatched drivers $\mathcal{D}$ and riders $\mathcal{R}$, their bids $\{b_d,\,d\in\mathcal{D}\}$ and $\{\delta_r,\,r\in\mathcal{R}\}$. The platform then needs to determine (1) The matched pairs of riders and drivers; and (2) the prices (charges and payments) of each match.

For a potential match $l=(d,r)$ with given bids $\delta_r$ and $b_d$, the social welfare is
\begin{equation}\label{socwell}
\sigma_{dr} =P_r(\delta_r) -P_d(b_d),
\end{equation}
where $P_r(\delta_r)$ and $P_d(b_d)$ are the participants' valuations of the matched trip $(d,r)$; see \eqref{Ud} and \eqref{Ur}.  The platform’s payment to the driver $d$ is denoted $q_d$, and the charge to the rider $r$ is $q_r$. Their utilities are respectively defined as:
\begin{equation}\label{utilityd}
u_d =q_d-P_d(b_d),
\end{equation}
\begin{equation}\label{utilityr}
u_r =P_r(\delta_r)-q_r.
\end{equation}
In other words, the driver's utility is the compensation paid by the platform minus her valuation of the trip; the rider's utility is her valuation of the trip minus the price paid to the platform.

\subsection{The VCG matching and pricing scheme}\label{subsecVCG}

In the classical VCG scheme, the matching decision is driven by social welfare maximization over all matched pairs:
\begin{equation}\label{vcgmopt1}
\hbox{[Social welfare]}~~  \max_{{\bf x}=(x_{dr})_{d\in\mathcal{D}, r\in\mathcal{R}}} V=\sum_{r\in\mathcal{R}} \sum_{d\in\mathcal{D}} \sigma_{dr} x_{dr}
\end{equation}
\begin{equation}\label{vcgmopt2}
\hbox{s.t.}~\sum_{r\in \mathcal{R}}x_{dr}\leq 1,~\forall d\in \mathcal{D}
\end{equation}
\begin{equation}\label{vcgmopt3}
\sum_{d\in \mathcal{D}}x_{dr}  \leq 1,~\forall r\in\mathcal{R} 
\end{equation}
\begin{equation}\label{vcgmopt4}
x_{dr} \in\left \{ 0 ,1\right \}, ~\forall d\in\mathcal{D},\,r\in\mathcal{R}
\end{equation}
\noindent Here, the binary decision variables $x_{dr}$ is 1 if driver $d$ and rider $r$ are matched. The optimal solution is denoted ${\bf x}^*=\{x_{dr}^*\}$ and $V^*$.

Based on the matching module \eqref{vcgmopt1}-\eqref{vcgmopt4}, we define $V_{d-}$ and $V_{r-}$ to be the maximum social welfare after removing driver $d$ and rider $r$ from the pool, respectively. Then, the VCG pricing scheme works as follows: For driver $d$, a non-negative bonus $V^*-V_{d-}$ is added to the bidding valuation $P_d(b_d)$ to form the final payment; for rider $r$, a non-negative subsidy of $V^*-V_{r-}$ is subtracted from the bidding valuation $P_r(\delta_r)$ to form the final charge. The following algorithm explains such a procedure. Here, all relevant quantities are from a single decision epoch. 

	\setlength\LTleft{0pt}
	\setlength\LTright{0pt}
	\begin{longtable}{@{\extracolsep{\fill}}rl}
	\hline
	\multicolumn{2}{l}{\bf Algorithm 1: VCG-based matching and pricing strategy}
	\\
	\hline
	{\bf Input}  & Set of unmatched drivers $\mathcal{D}$ and riders $\mathcal{R}$, their bids $\{b_d,\,d\in\mathcal{D}\}$ and $\{\delta_r,\,r\in\mathcal{R}\}$,
	\\
	&   riders' O-D information.	
 	\\
{\bf Step 1} & Solve \eqref{mopt1}-\eqref{mopt5} with social welfare as objective to obtain the matching solution  
        \\
        & ${\bf x}^*$ and the maximum social welfare $V^*$. Let $\mathcal{D}^*$ and $\mathcal{R}^*$ be the sets of matched 
        \\
        & drivers and riders, and $U$ be the total drive-by sensing externality.
        \\
{\bf Step 2}    &For every driver $d\in \mathcal{D}^*$,
 	\\
 	& \quad \quad calculate $V_{d-}$ by maximizing social welfare in \eqref{mopt1}-\eqref{mopt5} with $d$ removed;
 	\\
 	& \quad \quad the bonus of driver $d$: $\rho_d=V^*-V_{d-}$;
 	\\
 	& \quad \quad payment to driver $d$: $q_{d} =P_d +\rho_d$.
 	\\
 	& End for
 	\\
 	&For every rider $r\in \mathcal{R}^*$
 	\\
 	& \quad \quad calculate $V_{r-}$ by maximizing social welfare in \eqref{mopt1}-\eqref{mopt5} with $r$ removed;
 	\\
 	& \quad \quad the bonus of rider $r$: $\rho_r=V^*-V_{r-}$;
 	\\
 	& \quad \quad charge to rider $r$: $q_r =P_r -\rho_r$.
 	\\
 	&End for
 	\\
 {\bf Output}	&The matching solution ${\bf x}^{*}$, pricing solution ${\bf q}^*$, maximum social welfare $V^*$, and 
 \\
 & sensing externality $U$. 
 	\\
	\hline
	\end{longtable}

\subsection{The drive-by sensing (DS) matching and pricing scheme}\label{subsecDSdesc}

The DS matching scheme focuses on the sensing gain $\zeta_{dr}$ \eqref{zetadr} contributed by a matched trip $(d, r)$, given that its route is known based on the origin and destination of the requested trip. In contrast to the VCG scheme, which maximizes the social welfare, the DS matching scheme aims to maximize the total sensing externality:
\begin{equation}\label{mopt1}
\hbox{[Sensing utility]}~~   \max_{{\bf x}=(x_{dr})_{d\in\mathcal{D}, r\in\mathcal{R}}} U=\sum_{r\in\mathcal{R}} \sum_{d\in\mathcal{D}} \zeta_{dr} x_{dr}
\end{equation}
\begin{equation}\label{mopt2}
\hbox{s.t.}~\sum_{r\in \mathcal{R}}x_{dr}  \leq 1,\quad \forall d\in\mathcal{D}
\end{equation}
\begin{equation}\label{mopt3}
\sum_{d\in\mathcal{D}}x_{dr}  \leq 1,\quad\forall r\in\mathcal{R}
\end{equation}
\begin{equation}\label{mopt4}
\sum_{r\in\mathcal{R}}\sum_{d\in\mathcal{D}}\sigma_{dr}x_{dr}   \geq 0  
\end{equation}
\begin{equation}\label{mopt5}
x_{dr} \in \{0, 1 \} , \quad \forall d\in\mathcal{D},\,r\in\mathcal{R}
\end{equation}

\noindent  Besides the maximization objective, another important distinction from the VCG scheme is constraint \eqref{mopt4}, which stipulates that the total social welfare of all matched trips is non-negative. This is to ensure individual rationality of the subsequent pricing scheme, see Proposition \ref{probIR}. Note that the presence of the opportunity cost in the driver's trip valuation $P_d(b_d)$ could render negative social welfare $\sigma_{dr}$, which is allowed here in the interest of sensing gain, while such a pair $(d, r)$ will be eliminated in the VCG scheme because it introduces negative social welfare.

Following the matching module, the DS pricing scheme works as follows. After receiving the participants' bids, the platform solves the matching problem \eqref{mopt1}-\eqref{mopt5} for ${\bf x}^*$ and $U^*$. Next, let $V$ be the total social welfare corresponding to the matching solution ${\bf x}^*$, which  is non-negative according to \eqref{mopt4}. Then, $V$ is distributed to the participants according to their contributions to the sensing externality, thereby internalizing sensing externality. The following algorithm elaborates such a matching and pricing scheme.

\setlength\LTleft{0pt}
	\setlength\LTright{0pt}
	\begin{longtable}{@{\extracolsep{\fill}}rl}
	\hline
	\multicolumn{2}{l}{\bf Algorithm 2: Drive-by sensing matching and pricing scheme}
	\\
	\hline
	{\bf Input}  & Set of unmatched drivers $\mathcal{D}$ and riders $\mathcal{R}$, their bids $\{b_d,\,d\in\mathcal{D}\}$ and $\{\delta_r,\,r\in\mathcal{R}\}$,
	\\
	&   riders' O-D information.	 
	\\
	{\bf Step 1} & Solve \eqref{mopt1}-\eqref{mopt5} with sensing externality as objective to obtain the matching   
        \\
        & solution ${\bf x}^*$ and the maximum sensing externality $U^*$. Let $\mathcal{D}^*$ and $\mathcal{R}^*$ be the sets 
        \\
        & of matched drivers and riders, and $V$ be the total social welfare.
        \\
{\bf Step 2}    & For every driver $d\in \mathcal{D}^*$
 	\\
 	& \qquad calculate $U^*_{d-}$ by maximizing sensing externality in \eqref{mopt1}-\eqref{mopt5} with $d$
 	\\
 	& \qquad  removed; the contribution to sensing externality of $d$: $\Delta U_d=U^*-U^*_{d-}$;
 	\\
 	& End for
 	\\
    & For every rider $r\in \mathcal{R}^*$
 	\\
 	&\qquad  calculate $U_{r-}^*$ by maximizing sensing externality in \eqref{mopt1}-\eqref{mopt5} with $r$ 
 	\\
 	&\qquad  removed; the contribution to sensing externality of $r$: $\Delta U_r=U^*-U^*_{r-}$.
 	\\
 	& End for
 	\\
  {\bf Step 3}  &For every driver $d\in \mathcal{D}^*$
 	\\
 	&\qquad calculate share of the sensing externality $\displaystyle \lambda_{d} ={\Delta U_d \over \sum\limits_{d'\in\mathcal{D}^*}\Delta U_{d'} +\sum\limits_{r'\in\mathcal{R}^*} \Delta U_{r'}}$;
 	\\
 	& \qquad the bonus of driver $d$: $\rho_d =V\cdot \lambda_d$;
 	\\
 	& \qquad payment to driver $d$: $q_d =P_d +\rho_d$;
 	\\
 	& End for
 	\\
	&For every rider $r\in \mathcal{R}^*$
 	\\
 	&\qquad calculate share of the sensing externality $\displaystyle \lambda_{r} ={\Delta U_r \over \sum\limits_{d'\in\mathcal{D}^*}\Delta U_{d'} +\sum\limits_{r'\in\mathcal{R}^*} \Delta U_{r'}}$;
 	\\
 	& \qquad the bonus of rider $r$: $\rho_r =V\cdot \lambda_r$;
 	\\
 	& \qquad charge to rider $r$: $q_r =P_r -\rho_r$.
 	\\
 	& End for
 	\\
{\bf Output} & The matching solution ${\bf x}^{*}$, pricing solution ${\bf q}^*$, maximum sensing externality $U^*$  
\\
& and total social welfare $V$.
 	\\
	\hline
	\label{longtab2}
	\end{longtable}

\subsection{Properties of the matching and pricing schemes}\label{subsecDSprop}

The following properties of a matching and pricing scheme are frequently discussed.

\begin{definition}
Let ${\bf y}^*$ be an optimal solution of \eqref{mopt1}-\eqref{mopt5}, $\mathcal{D}^*$ and $\mathcal{R}^*$ be the set of matched drivers and riders, respectively. A pricing solution is called ideal if and only if the following four conditions are met:
\begin{itemize}
\item[(1)] Budget balancing (BB) or weakly budget balancing (WBB): The total payment to the drivers is no greater than the total charge to the riders:
\begin{equation}\label{BB}
\text{BB:}~\sum_{d\in\mathcal{D}^*}q_d=\sum_{r\in\mathcal{R}^*}q_r,
\qquad
\text{WBB:}~\sum_{d\in\mathcal{D}^*}q_d\leq \sum_{r\in\mathcal{R}^*}q_r.
\end{equation}
\item[(2)] Individually rational (IR): Any driver receives a payment that is no less than his/her valuation of the trip; any rider pays no more than his/her own valuation of the trip.
\begin{equation}\label{IR}
q_d \geq P_d(b_d)~~\forall d\in\mathcal{D}^*,\qquad q_r \leq P_r(\delta_r)~~\forall r\in\mathcal{R}^*.
\end{equation}
\item[(3)] Incentive compatible (IC): No participant can unilaterally increase their utility by misreporting their bidding valuations $b_d$ or $\delta_r$.
\item[(4)] Allocative efficient (AE): The outcome of the matching and pricing scheme reaches maximum utility for the system, which can take the form of total social welfare \eqref{vcgmopt1} or sensing gain \eqref{mopt1}. 
\end{itemize}
\end{definition}

It is widely known that the VCG scheme satisfies AE, IR and IC, but does not guarantee BB. In fact, the total deficit that the platform incurs is equal to the total social welfare minus the participants' bonuses:
$$
\sum_{\{d, r\vert x_{dr}^*=1\}}\big(\sigma_{dr}-\rho_d-\rho_r\big) =V^*-\sum_{d\in\mathcal{D}^*}\rho_d - \sum_{r\in\mathcal{R}^*}\rho_r.
$$

In the rest of this section, we discuss the properties of the proposed DS matching and pricing scheme.

\begin{proposition}{\bf (AE)} The DS scheme is allocative efficient. 
\end{proposition}
\begin{proof}
According to {\bf Step 1}, the matched pairs are such that the total sensing externality \eqref{mopt1} is maximized. 
\end{proof}

\begin{proposition}{\bf (BB)} The DS scheme is budget balancing.
\end{proposition}
\begin{proof}
Recall that $q_d$ and $q_r$ are the platform's payments (charges) to the drivers (riders). Therefore, the platforms' total revenue is: 
\begin{align*}
\sum_{r\in \mathcal{R}^*} q_r -\sum_{d\in\mathcal{D}^*} q_d=&\sum_{r\in \mathcal{R}^*} (P_r-\rho_r)-\sum_{d\in \mathcal{D}^*} (P_d+\rho_d)
\\
=&\sum_{r\in\mathcal{R}}\sum_{d\in\mathcal{D}}x^*_{dr}\big(P_r-P_d\big) - \sum_{r\in\mathcal{R}^*}\rho_r -\sum_{d\in\mathcal{D}^*}\rho_d
\\
=& V - V\left(\sum_{r\in\mathcal{R}^*}\lambda_r + \sum_{d\in\mathcal{D}^*}\lambda_d \right)
\\
=& V- V\cdot {\sum_{r\in\mathcal{R}^*}\Delta U_r + \sum_{d\in\mathcal{D}^*}\Delta U_d \over \sum\limits_{d'\in\mathcal{D}^*}\Delta U_{d'} +\sum\limits_{r'\in\mathcal{R}^*} \Delta U_{r'}}=0.
\end{align*}
\end{proof}
\begin{remark}
A minor tweak to the DS pricing scheme is introduced in Section \ref{subsecdetail} to render more reasonable charges in certain extreme cases. Such a modification in fact makes the scheme weakly budget balancing, i.e. the platform's revenue becomes positive.
\end{remark}

\begin{proposition}\label{probIR}{\bf (IR)} The DS scheme is individually rational 
\end{proposition}
\begin{proof}
According to constraint \eqref{mopt4}, the total social welfare $V\geq 0$ following {\bf Step 1} of the algorithm. Obviously, the sensing externality share $\lambda_d,\,\lambda_r \geq 0$ for all $d\in\mathcal{D}^*$ and $r\in\mathcal{R}^*$. Therefore, the bonuses $\rho_d=V\cdot\lambda_d$ and $\rho_r=V\cdot \lambda_r$ are all non-negative. This finishes the proof.
\end{proof}

The DS scheme is only partially incentive compatible as it does not rule out over-reporting, as characterized by the following results.  

\begin{lemma}
Under the DS policy, when driver $d$ (rider $r$) truthfully reports $\bar b_d$ ($\bar \delta_r$) and was matched to no one, she cannot falsely reports $b_d$ ($\delta_r$) to increase her utility. 
\end{lemma}
\begin{proof}
Since the scheme matches participants by maximizing total sensing gain, which depends only on their locations and O-D information, rather than their biddings. Therefore, when a participant is not matched, falsely reporting $b_d$ or $\delta_r$ cannot alter her matching status and hence, her utility.
\end{proof}

\begin{lemma}
Under the DS policy, when a driver $d$ (rider $r$) truthfully reports $\bar b_d$ ($\bar \delta_r$) and was matched, she cannot increase her utility by under-reporting $b_d < \bar b_d$ ($\delta_r<\bar \delta_r$); however, she may do so by over-reporting $b_d > \bar b_d$ ($\delta_r>\bar \delta_r$). 
\end{lemma}
\begin{proof}
The utility of the participant is expressed as
$$
u_d=q_d - P_d(b_d)\quad \forall d\in \mathcal{D}^*,\qquad u_r=P_r(\delta_r)-q_r\quad \forall r\in\mathcal{R}^*.
$$
When a matched driver $d$ falsely reports, the change of utility relative to the truthful one (distinguished using $\bar{\cdot}$) becomes:
\begin{align*}
u_d-\bar u_d 
&= P_d + V\lambda_d -\bar P_d - \bar V \lambda_d = P_d-\bar P_d +\lambda_d(V-\bar V)
\\
& = P_d - \bar P_d +\lambda_d(\sigma_{dr}-\bar \sigma_{dr})= P_d-\bar P_d - \lambda_d(P_d-\bar P_d)
\\
& =  (P_d-\bar P_d)(1-\lambda_d) = (\tau_{dr}-\tau_d^{\text{min}})(b_d -\bar b_d)(1-\lambda_d).
\end{align*}
Similarly, when a matched rider $r\in\mathcal{R}^*$ falsely report, the change of utility is:
\begin{align*}
u_r-\bar u_r &= \bar P_r- (P_r - V\lambda_r) - \bar V\lambda_r= \bar P_r -P_r -\lambda_r (\bar V -V)
\\
& = \bar P_r -  P_r - \lambda_r(\bar \sigma_{dr} - \sigma_{dr})= \bar P_r-  P_r - \lambda_r( \bar P_r - P_r)
\\
&= (\bar P_r-P_r)(1-\lambda_r)=  (\tau_{dr}-\tau_r^{\text{min}})(\delta_r-\bar \delta_r)(1- \lambda_r).
\end{align*}
Since $\tau_{dr}-\tau_d^{\text{min}}\geq 0$, $1-\lambda_d\geq 0$, $\tau_{dr}-\tau_r^{\text{min}}\geq 0$ and $1-\lambda_r\geq 0$, we conclude that $u_d-\bar u_d$ (respectively $u_r-\bar u_r$) and $b_d -\bar b_d$ (respectively $\delta_r-\bar \delta_r$) have the same sign. This completes the proof.
\end{proof}

Although the DS matching and pricing scheme does not prevent individuals to over-report, the next result shows that the sum of the participants' utilities is a constant, regardless of their falsely-reporting behavior. Such a property is termed {\it group incentive compatible} (G-IC), meaning that as a group, falsely reporting does not raise the total utility of the entire group.  In the following, symbols with $\bar \cdot$ always correspond to truthful reporting. 

\begin{proposition}{\bf (G-IC)}\label{propGIC}
Assume the biddings of driver $d$ and rider $r$ are: $b_d=\bar b_d +\varepsilon_d$, $\delta_r=\bar\delta_r+\varepsilon_r$, where $\varepsilon_d, \varepsilon_r\in \mathbb{R}$. Then, the total utility $\sum_{d\in\mathcal{D}^*} u_d+\sum_{r\in\mathcal{R}^*}u_r$ of matched participants is a constant, equal to the total social welfare $\bar V$ without any over-reporting.
\end{proposition}

\begin{proof}
For a matched pair $(d, r)$ such that $x_{dr}^*=1$, the amount of social welfare reduction due to over-reporting is denoted
$$
\Delta \sigma_{dr}\doteq (\tau_{dr}-\tau_d^{\text{min}})\varepsilon_d + (\tau_{dr}-\tau_r^{\text{min}})\varepsilon_r.
$$
\noindent Therefore, the total social welfare, resulting from over-reporting, is expressed as
$$
V=\bar V -\sum_{\{d, r\vert x_{dr}^*=1\}}\Delta\sigma_{dr}.
$$
We deduce that:
\begin{align*}
&\sum_{d\in\mathcal{D}^*} u_d+\sum_{r\in\mathcal{R}^*}u_r
\\
=&\sum_{\{d, r\vert x_{dr}^*=1\}}u_d+u_r= \sum_{\{d, r\vert x_{dr}^*=1\}} \Big[\lambda_d V +  (\tau_{dr}-\tau_d^{\text{min}}) \varepsilon_d +\lambda_r V + (\tau_{dr}-\tau_r^{\text{min}})\varepsilon_r \Big]
\\
=&\sum_{\{d, r\vert x_{dr}^*=1\}}\Big[ (\lambda_d+\lambda_r) V+ \Delta\sigma_{dr} \Big] = V \sum_{\{d, r\vert x_{dr}^*=1\}}(\lambda_d+\lambda_r) + \sum_{\{d, r\vert x_{dr}^*=1\}}\Delta\sigma_{dr}
\\
=& V + \sum_{\{d, r\vert x_{dr}^*=1\}}\Delta\sigma_{dr}=\bar V.
\end{align*}
\noindent Here, we make use of the fact that the individual shares of the total social welfare, $\lambda_d$ or $\lambda_r$, add up to one.
\end{proof}

\begin{remark}
The significance of Proposition \ref{propGIC} is the fact that, for a group, the DS scheme is incentive compatible, despite its failure on an individual level. Note that this is also a rather general result, as no assumptions are made on the signs (i.e. works for both under-reporting and over-reporting), magnitude, or the individual heterogeneity of $\varepsilon_d$ and $\varepsilon_r$. Any individual or collective behavior regarding false reporting would not change the total utility of the group.
\end{remark}

Finally, we show that the outcome of the DS matching and pricing scheme is envy-free, which means drivers at the same location do not envy each other's future payoffs, and riders requesting the same trip do not envy each other's outcomes \citep{MFP2022}. Envy-freeness is an important aspect of fairness that is critical to the long-term health of a marketplace. 

\begin{definition}{\bf (Envy-free)} 
A scheme is envy-free for riders if any two riders $r, r'\in\mathcal{R}$ making the same trip requests (i.e. $s_{r}=s_{r'}$, $t_r=r_{r'}$ within the same decision epoch) have the same utility $u_r=u_{r'}$. A scheme is envy-free for drivers if any two drivers $d, d'\in\mathcal{D}$ submitting the same location to the platform within the same decision epoch have the same utility $u_d=u_{d'}$. 
\end{definition}

\begin{proposition}
The DS scheme is envy-free. 
\end{proposition}
\begin{proof}
According to Algorithm 2, the utilities of a rider (or driver) is equal to her bonus $\rho_r$ (or $\rho_d$) if she is matched, and zero otherwise. Note that the DS-oriented matching plan \eqref{mopt1}-\eqref{mopt5} depends solely on the participants' locations and trip O-Ds, and the pricing scheme allocates the total social welfare $V$ by the individual's share of sensing contribution $\lambda_r$ or $\lambda_d$. Therefore, any pair of riders $r, r'$ making the same trip request at the same time have the same matching outcome, and their shares satisfy $\lambda_r=\lambda_{r'}$, hence $\rho_r=\rho_{r'}$. Similarly, any pair of drivers $d, d'$ at the same location and time have the same matching outcome and share $\lambda_d=\lambda_{d'}$, thus the same bonus $\rho_d=\rho_{d'}$. 
\end{proof}

\subsection{Implementation details of the matching and pricing schemes}\label{subsecdetail}
A few more techniques and details used in the implementation of the proposed scheme are described below. 
\begin{enumerate}
\item For rider-driver matching, a maximum search radius $R$ (km) is prescribed, such that a driver and rider over $R$ km apart are not considered for matching. In the numerical experiment below, $R=2$ (km).

\item In the DS scheme, the calculation of sensing externality of a matched trip $(d, r)$ does not include its pick-up phase. This ensures good service quality by avoiding unnecessarily long pick-up distances in the pursuit of sensing gain. In other words, the sensing gain is  solely based on the riders' requested trips, irrespective of the drivers locations.

\item It is possible for a trip to contribute significant sensing gain, such that the rider's bonus $\rho_r$ might exceed her willingness to pay $P_r$, leading to negative charges $q_r=P_r-\rho_r$, which is unreasonable. In this case, a lower bound $\alpha h_r$ on the charges will apply:
$$
q_r=\max\big(P_r-\rho_r,\, \alpha h_r \big),
$$ 
\noindent where $\alpha$ is the driver's cost per unit travel distance, and $h_r$ is the total distance traveled from the origin to the destination of the requested trip. The bonus saved: $\max (0,\, \alpha h_r -(P_r-\rho_r))$ becomes the platform's revenue, leading to weakly budget balance. 

\end{enumerate}

\section{Numerical study}\label{secSim}
This section offers an elaborated numerical study of the proposed scheme, starting with an illustrative example of relevant notions (Section \ref{subsecillex}), followed by a simulation-based comparative study that highlight the key characteristics and benefits of the proposed scheme (Section \ref{subsecSS}).

\subsection{Illustrative example}\label{subsecillex}

This part offers an intuitive example relating to the two matching and pricing schemes mentioned in this work. We also demonstrate the opportunity cost and how it might affect drivers' willingness to accept orders.

We begin with the order prospect function \eqref{prodef}, where the probability densities $n_g$'s are obtained directly from historical data (see Figure \ref{figprospects}a), and the weight $w_{g,t_r}$ is instantiated as a linear function of the distance $d_{g,t_r}$ from grid $g$ to the destination $t_r$:
$$
w_{g, t_r}= 1-d_{g,t_r}/M \quad \hbox{so that}\quad p(t_r)=\sum_{g\in G}\left(1-{d_{g, t_r}\over M} \right)n_g,
$$
\noindent where $M=\max_{g\in G}\left \{ d_{g,t_{r}} \right \}$. As the opportunity cost should be a monotonically decreasing function of $p(t_r)$, we consider the following form:
\begin{equation}\label{oppocostdef}
f(t_r)=f\big(p(t_r)\big)=
\begin{cases}
\xi \big(p^*-p(t_r)\big)  \qquad & p(t_r)\in [p^{\text{min}},\, p^*) 
\\
0 \qquad & p(t_r)\in [p^*,\, p^{\text{max}}]
\end{cases}
\end{equation}
\noindent where $\xi>0$, $p^{\text{min}}$ and $p^{\text{max}}$ are the minimum and maximum prospects over all grids, $p^*$ is a critical value. This formulation means that if the destination $t_r$ is in a grid with prospect greater than $p^*$, then the opportunity cost is zero; otherwise, the opportunity cost grows as the order prospect decreases. In the following numerical case study, $\xi=50$, and $p^*=0.9 p^{\text{max}}$.

We use Figure \ref{figprospects} to visualize the distribution of trip requests, order prospect $p(\cdot)$, and opportunity cost $f(\cdot)$, based on real-world data in Longquanyi District, Chengdu. As shown in Figure \ref{figprospects}(a), most of the trip requests are concentrated in the mid-east region, and the order prospect decays as we move away from this region. Moreover, the northern part, which is very far away from the mid-east region and has very few requests, ends up with very low order prospects (Figure \ref{figprospects}b) and high opportunity cost (Figure \ref{figprospects}c). For example, if a driver drops off a passenger at the northern region, she may incur a significant cost while cruising for the next order.

\begin{figure}[h!]
\centering
\includegraphics[width=\textwidth]{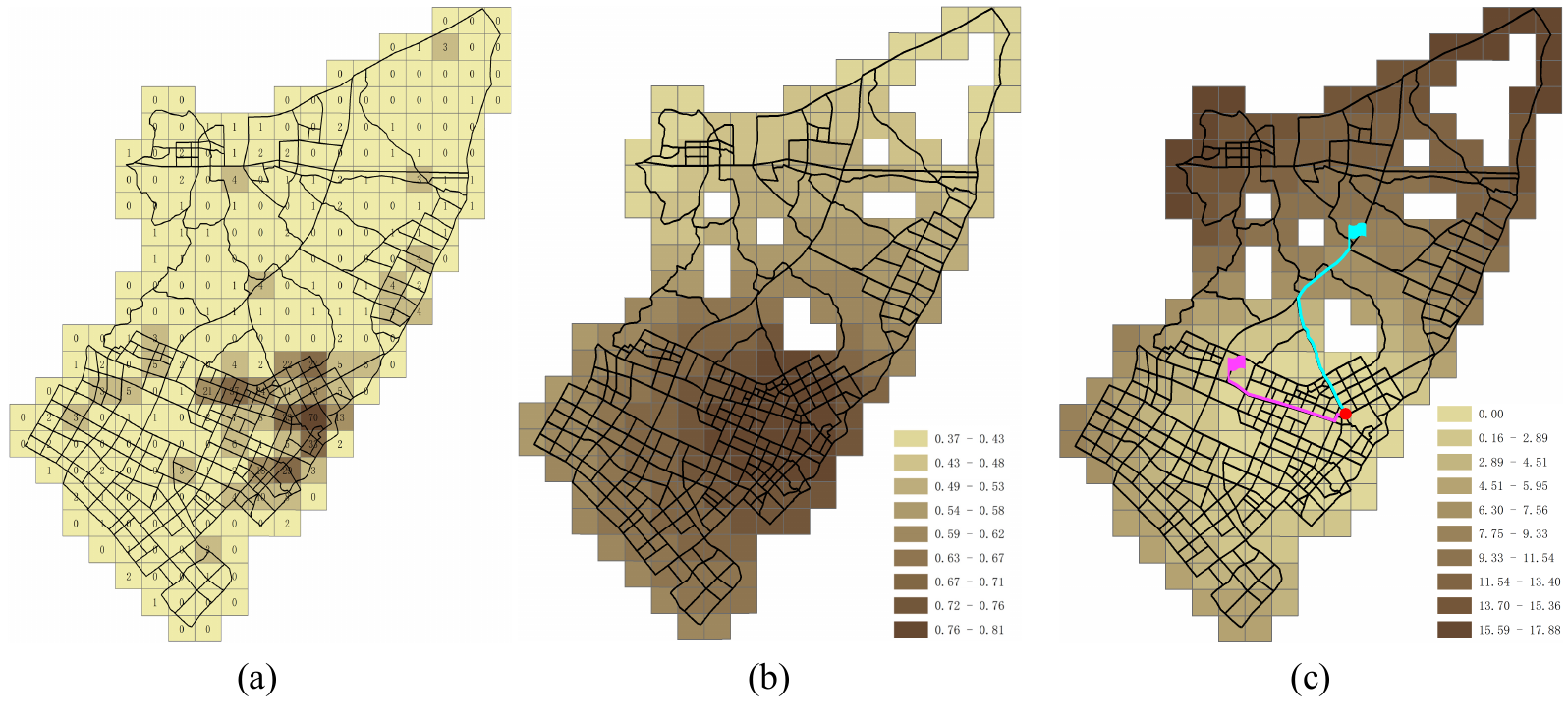}
\caption{Spatial distribution of number of trip requests (a), order prospect (b), and opportunity cost (c).}
\label{figprospects}
\end{figure}

We consider two riders $r_1, r_2$ with the same origin marked as a red dot, and distinct destinations marked as two flags in Figure \ref{figprospects}(c). The shortest routes serving these trips are marked for $r_1$ (7.2 km) and $r_2$ (4.8 km). The opportunity costs associated with the destinations are: $f(t_{r_1})=7.56$, $f(t_{r_2})=0$.

The monetary value per unit distance traveled (CNY/km) are respectively $\alpha=1.5$ for drivers and $\beta=2.75$ for riders. A vacant taxi $d$ is 0.5 km away from both riders, thus $\tau_{dr_1}=\tau_{dr_2}=\tau_r^{\text{min}}$. According to \eqref{Ud} and \eqref{Ur}, the participants' valuations are:
\begin{align*}
\text{match}~(d,\,r_1):& ~\begin{cases}
P_d=1.5\times7.2+ 0 + 7.56=18.36
\\
P_{r_1}=2.75\times 7.2 +0 =19.80
\end{cases}\quad\Longrightarrow \sigma_{dr_1}=1.44;
\\
\text{match}~(d,\,r_2): & ~
\begin{cases}
P_d=1.5\times 4.8+0+0=7.2
\\
P_{r_2}=2.75\times 4.8 +0  =13.20
\end{cases}\quad\Longrightarrow \sigma_{dr_2}=6.
\end{align*}

\noindent The VCG scheme selects the match $(d,\,r_2)$ for larger social welfare, while the DS scheme choses $(d,\,r_1)$ because the trip of $r_1$ contributes more sensing gain as it is longer and covers areas with low trip demands.

\subsection{Simulation study}\label{subsecSS}

\subsubsection{Key performance areas and indicators}\label{subsubsecKPA}
This section evaluates the VCG and DS matching and pricing schemes in terms of three key performance areas (KPAs) and five key performance indicators (KPIs):
\begin{itemize}
\item {\bf Level of service}: (1) {\it Matching rate}, which is the percentage of riders that are matched to drivers; (2) {\it Average wait time}, which is the averaged pick-up times $\tau_{dr}$ of matched pairs. 
\item {\bf Sensing externality}: (3) {\it Sensing utility} $\Phi$, which is given in \eqref{eqnPhi}; (4) {\it Grid coverage rate}, which is the average \% of grids covered at least once in an hour. 
\item (5) {\bf Platform revenue}, which is the total charge to the riders minus the total payment to the drivers. 
\end{itemize}

\subsubsection{Simulation setup}
The simulation period spans four hours, with a total of four 1-hr sensing intervals. Within each sensing interval, there are 18 decision epochs (each of length 3 min 20s).

The trip demands are generated using historical information on completed trips between 8:00-12:00, collected from 1 Aug to 31 Dec, 2021. Three demand scenarios are considered, with increasing proportion of remote trip (i.e. trips ending with low local trip demands), as shown in Figure \ref{figlmh}. In all three scenarios, the total number of trip requests is set constant (around 144 per hour).

At the beginning of each simulation run the taxis are randomly positioned in the network. Vacant taxis move towards areas with high order prospects (Figure \ref{figprospects}b) until they are assigned orders. All taxis travel at a constant speed of 35 km/hr. 

\begin{figure}[h!]
\centering
\includegraphics[width=\textwidth]{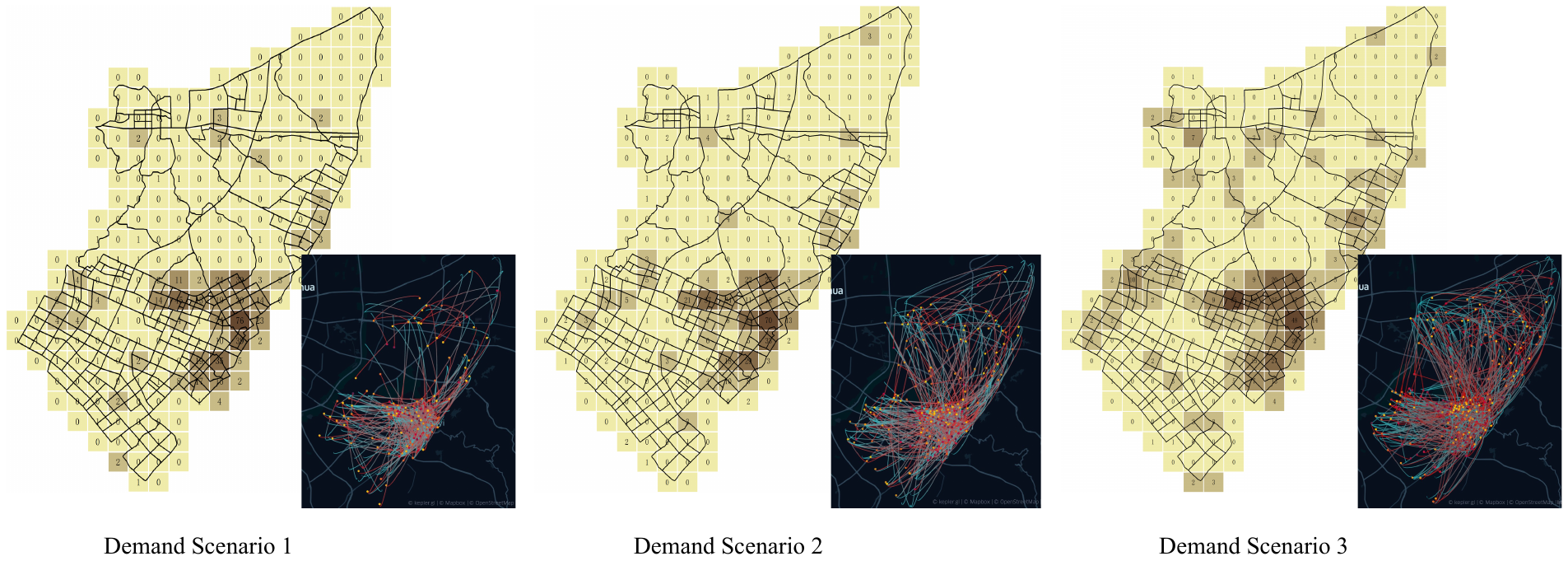}
\caption{The three demand scenarios, with increasing percentage of remote order number.}
\label{figlmh}
\end{figure}

\subsubsection{Simulation results on the KPIs}
The performances of VCG and DS schemes, in terms of the key performance areas presented in Section \ref{subsubsecKPA}, are compared in Figure \ref{figsim3cases} for the three demand scenarios, as well as fleet size varying from 20 to 60. The participants' true bid value of the extra pick-up distance, $b_d$ and $\delta_r$, are drawn from a uniform distribution $U[1,\,2]$ (CNY/km). 

\begin{figure}[H]
\centering 
\includegraphics[width=\textwidth]{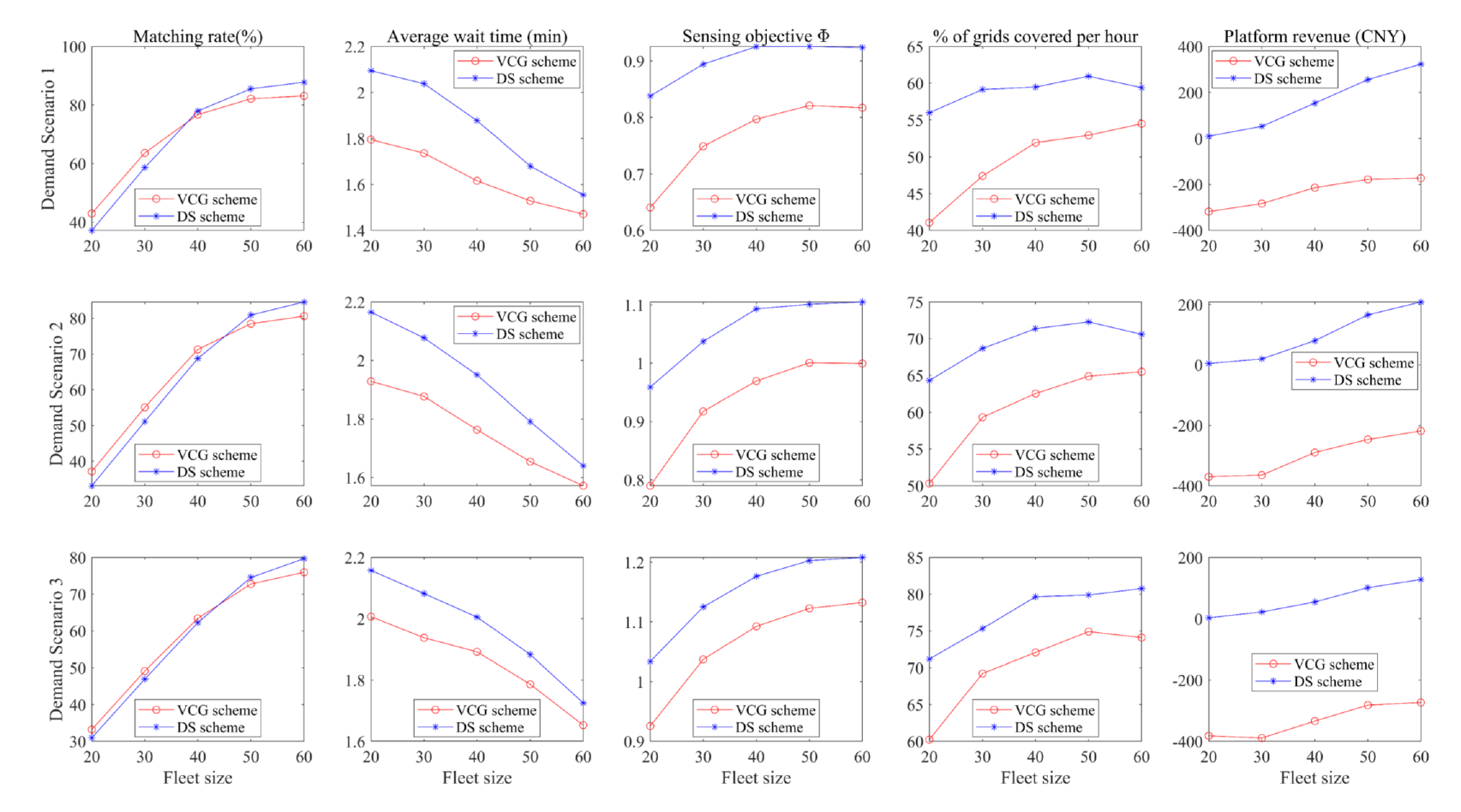}
\caption{Performance comparison of VCG and DS matching \& pricing schemes. Each row corresponds to a demand scenario.}
\label{figsim3cases}
\end{figure}

\noindent The following observations are made from Figure \ref{figsim3cases}:
\begin{enumerate}
\item The level of service (matching rate \& wait time) and sensing externality (sensing utility \& grid coverage rate) both improve as the fleet size increases. 

\item The VCG and DS schemes have similar matching rates. The former slightly outperforms the latter when the fleet size $\leq 40$, and the situation is reversed for fleet size $>40$. This is because, given abundant supply (fleet size $>40$), matches with negative social welfare and high DS externality will be pursued by the DS scheme, but abandoned by the VCG scheme.

\item The VCG scheme yields less wait time than the DS scheme, and the difference is larger for smaller fleet size (up to 0.3 min). The reason is that the VCG scheme tends to favor nearest matches as the social welfare is maximized in this way. In contrast, the DS scheme first selects riders whose requested trips maximize the sensing utility, and subsequently match them with drivers within a 2-km radius, ending up with higher wait time. 

\item In terms of sensing objective and grid coverage, the DS scheme considerably outperform the VCG scheme. The gap between the two is larger in Demand Scenario 1, suggesting that the DS scheme is suited for highly unbalanced trip distributions.

\item The DS scheme yields positive revenue while the VCG scheme has a significant deficit. As pointed out at the end of Section \ref{subsecdetail}, when binding, the lower bounds on riders' charges generate revenue for the platform. Such revenue grows with the fleet size. The revenue is also higher when the demand distribution is highly unbalanced (e.g. in Demand Scenario 1). This is because when the trips are spatially unbalanced, those remote orders could have very high sensing gain, resulting in significant bonus on the riders' part. When this happens, the binding lower bound  on their charges generate more revenue for the platform. 
\end{enumerate}

The DS scheme allows a matched pair $(d,r)$ to have negative social welfare $\sigma_{dr}<0$, in exchange for sensing gain $\zeta_{dr}$. Such a trade-off between total social welfare and sensing utility is visualized in Figure \ref{figscatter}, where the scatter points represent matched trips, and their coordinates are $(\sigma_{dr},\,\zeta_{dr})$. 

\begin{figure}[H]
\centering
\includegraphics[width=\textwidth]{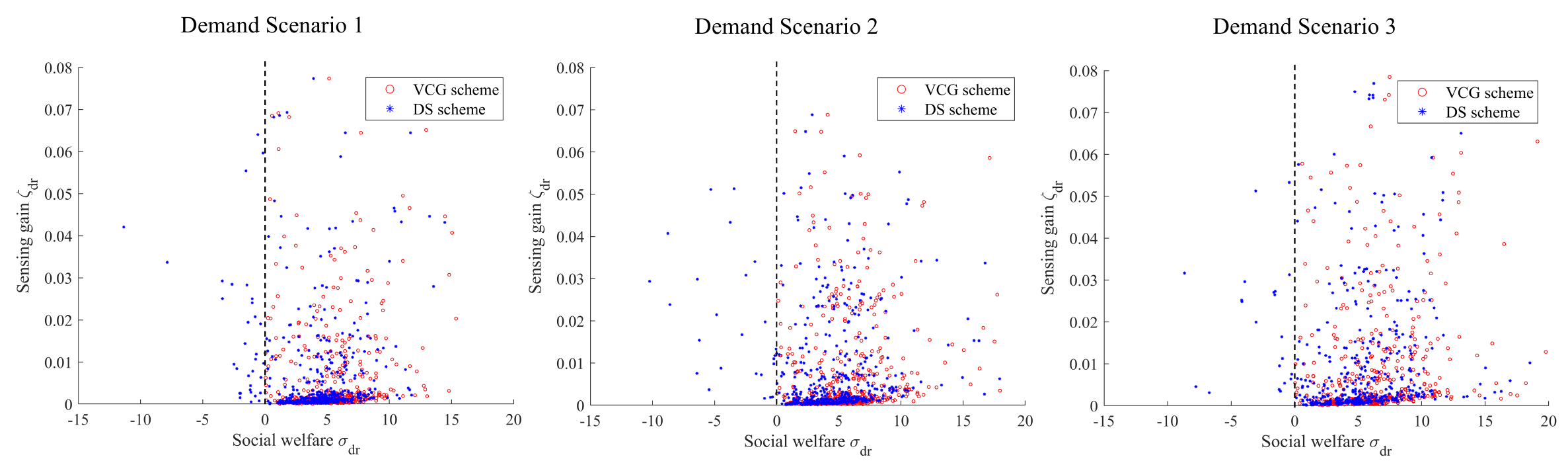}
\caption{Scatter plots of matched trips under the VCG and DS scheme(50 fleet size).}
\label{figscatter}
\end{figure}

\noindent It can be seen that: 
\begin{itemize}
\item The VCG scheme always yields matches with $\sigma_{dr}>0$; The DS scheme allows matches with negative $\sigma_{dr}$ and high $\zeta_{dr}$, while the total social welfare is still positive. 

\item The majority of the matches in either scheme have relatively low $\sigma_{dr}$ (between 0 and 10) and insignificant $\zeta_{dr}$ (below 0.005). These matches correspond to the trips starting and ending at populated areas (the mid-east region). 

\item The number of high-$\zeta_{dr}$ ($\geq 0.5$) matches  increases from Demand Scenario 1 to 3, and most of those matches have positive social welfare in Scenario 3. This means that the social welfare and sensing externality are better aligned when trip distributions are relatively balanced in space.
\end{itemize}

\subsubsection{Analysis on over-reporting}

As the DS scheme cannot rule out over-reporting, we conduct experiments assuming reasonable parameters regarding over-reporting to understand its impact on the participants' utilities and platform revenue. Specifically, individual biddings, which include over-reporting, are generated in the following way: 
$$
b_d=\bar b_d + \varepsilon,\quad \delta_r=\bar\delta_r+\varepsilon,
$$
where the true bidding values $\bar b_d, \bar\delta_r\sim U [1,\,2]$ (CNY), and the over-reporting part $\varepsilon\sim U [0,\,0.5]$ (CNY). In addition, we consider the following percentages of over-reporting participants: 0\%, 20\%, 40\% and 60\%. Table \ref{tabover} shows drivers' and riders' average utilities, which are defined in  \eqref{utilityd} and \eqref{utilityr}, as well as the platform's revenue, in various situations. These results are based on 10 independent simulation runs, each with randomly drawn $\bar b_d$, $\bar\delta_r$, and $\varepsilon$.

\begin{table}[!h]
	\centering
	\caption{Participants' average utilities and platform revenue under the DS scheme.}\label{revenueandutility20}  
	\resizebox{\columnwidth}{!}{
		\begin{tabular}{cc|ccc|ccc|ccc} 
\\ 
\hline
\multirow{2}{*}{\% of over} & \multirow{2}{*}{Demand} & \multicolumn{3}{c}{Fleet size: 20} & \multicolumn{3}{c}{Fleet size: 40}  & \multicolumn{3}{c}{Fleet size: 60} 
\\
\cline{3-11}
\multirow{2}{*}{reporting} &  \multirow{2}{*}{Scenario} & \multicolumn{2}{c}{Avg. utility} & \multirow{2}{*}{Revenue} & \multicolumn{2}{c}{Avg. utility} & \multirow{2}{*}{Revenue} & \multicolumn{2}{c}{Avg. utility} & \multirow{2}{*}{Revenue}    
\\ 
\cline{3-4}\cline{6-7}\cline{9-10}
                         &  & driver        & rider    &    & driver        & rider    &    & driver        & rider      &          
\\ 
\hline
\multirow{3}{*}{0\%} & 1 & 2.784       & 2.233         & 10.320   & 1.666      & 2.833   &157.084    & 1.148         & 3.175      & 299.689                       
\\
			       &2 & 3.207         & 2.556        & 5.157    & 2.205      & 3.110     &74.886       & 1.513          & 3.584       & 192.241    
\\
			       &3 & 3.715         & 3.022        & 4.633    & 2.752      & 3.445     &55.702       & 1.990        & 3.936        & 131.099                            
\\ 
\hline
\multirow{3}{*}{20\%}& 1 & 2.785     & 2.232         & 10.226   & 1.670      & 2.831     &155.769       & 1.153     &3.173    & 297.869 
\\
			       &2 & 3.208       & 2.555         & 5.139    & 2.209      & 3.108     & 74.287      & 1.518        &3.583     & 190.868 
\\
			       &3 & 3.717       & 3.021          & 4.585    & 2.755      &3.442     &55.481      & 1.992       & 3.934     & 130.594  
\\
\hline
\multirow{3}{*}{40\%} &1 & 2.786         & 2.231         & 10.174  &1.674      &2.829   &154.644      & 1.158         &3.171      & 296.270 
\\
			       &2 & 3.209           & 2.554         & 5.135    & 2.211      &3.107     &73.874       & 1.521           &3.580     & 189.959
\\
			       &3 & 3.718          & 3.020          & 4.564    &2.758      &3.440     &55.169     & 1.995     & 3.932    & 129.975
\\
\hline
\multirow{3}{*}{60\%} & 1 & 2.789       & 2.229          & 10.097  & 1.679     &2.827     &153.569      & 1.163          & 3.171      & 294.304
\\
			      &2 & 3.211            & 2.552          & 5.131    & 2.214     &3.105     &73.508      & 1.524        &3.580       & 188.733 
\\
			       &3 & 3.718          & 3.019          & 4.563     &2.761      &3.438     &54.905       &1.998       & 3.931        & 129.446
\\
\hline
	\end{tabular}
	\label{tabover}
	}
\end{table}

The following are observed, with each item focusing on the impact of one particular parameter:
\begin{enumerate}
\item {\bf Over-reporting \%}: As more participants over-report, the drivers' utilities increase, the riders' utilities decrease, and the platform's revenue decrease, all by very small margins. Such a trend is consistent across all three demand scenarios and fleet sizes. The reason is as follows. More over-reporting compresses the total social welfare and individuals' bonuses $\rho_d$ and $\rho_r$. According to item 3 of Section \ref{subsecdetail}, the platform’s revenue 
$$
\sum_{r\in\mathcal{R}^*} \max (0,\, \alpha h_r -(P_r-\rho_r))
$$
declines because of reduced $\rho_r$'s. As for the drivers, their average utility increases as a few drivers with zero DS contribution ($\lambda_d=0$) ended up gaining positive utilities because of over-reporting. And such gains on the drivers' part are paid by the riders, resulting in a decline in their average utility.

\item {\bf Trip distribution}: Regardless of fleet size or over-reporting \%, the participants' utilities all increase as the trip distribution becomes more balanced (from Demand Scenarios 1 to 3). This is due to the decrease of platform revenue, for the same reason explained in item 5 following Figure \ref{figsim3cases}.

\item {\bf Fleet size}: As the fleet size increases, the supply grows, leading to more competition among the drivers reducing their average utility; on the other hand, the rider's average utility increases because higher supply helps realize the requested trips' sensing potentials. In fact, $\lambda_r$'s are much greater than $\lambda_d$'s in case of high supply, because the non-zero sensing gains offered by requested trips general non-zero shares $\lambda_r$ for the riders, while many drivers have $\lambda_b=0$.
\end{enumerate}

\section{Conclusions}\label{secCon} 

This work explores the drive-by sensing (DS) potential of taxi fleets via e-hailing services. The primary shortcoming of taxi-based DS, which is spatially unbalanced distributions of vehicles (sensors), is addressed within an integrated matching and pricing scheme by prioritizing riders whose trips generate high DS gains, and rewarding drivers who deliver those gains. The proposed scheme is recapped below:
\begin{itemize}
\item The matching submodule, based on a sensing maximization principle, selects high-sensing-contribution trip requests, matched with near-by available drivers. As such a matching scheme is strongly oriented towards riders, which could lead to inefficiencies at the pick-up phase, a double auction mechanism is designed to manage the level of service. 

\item Building on the drivers' and riders' valuations of a potential match, which include both public and private information as well as opportunity cost, the pricing submodule allocates social welfare to participants according to their contribution to the sensing objective obtained from the matching phase. 
\end{itemize}

The proposed matching and pricing scheme is characterized as allocative efficient (sensing-oriented), individually rational and budget balancing (by allocating social welfare as bonuses), as well as group incentive compatible. The last property ensures that the cohort will always end up with the same total utility regardless of the mis-reporting behavior of its members.

To further confirm and quantify the said properties and benefits of the proposed framework, this paper conducts extensive simulation studies based on a real-world drive-by air quality sensing scenario. It is found that, 
\begin{enumerate}
\item compared with the VCG scheme as a benchmark, the proposed scheme not only excels in sensing efficacy, but also generates revenue for the platform, at very minor cost to the level of service (matching rate, wait time).

\item the relative performance also depends on a few factors, including fleet size, demand distribution, and over-reporting percentage. 

\item in a nutshell, the proposed scheme effectively trades social welfare for fulfilling high-sensing-gain trips, compressing user surplus for sensing externality; such a mechanism is quite effective in tested scenarios and robust against mis-reporting. 
\end{enumerate}

The practical significance of the proposed scheme is its operational feasibility, by relying on publicly available information, tolerating mis-reports, and offering a financially sustainable mode for the e-hailing platform. Future research will focus on: (1) more elaborate and empirical study on the concept of opportunity cost; (2) consideration of uncertain and heterogenous user behavior;  and (3) a holistic business model that involves users of drive-by sensing data.

\section*{Acknowledgement}
This work is partly supported by the National Natural Science Foundation of China through grant 72071163.


\begin{thebibliography}{100}
\bibitem[Agatz et al., 2011]{AESW2011} Agatz, N.A., Erera, A.L., Savelsbergh, M.W., \& Wang, X. (2011). Dynamic ride-sharing: A simulation study in metro Atlanta. Transportation Research Part B-methodological, 45(9), 1450-1464.

\bibitem[Agatz et al., 2012]{AESW2012} Agatz, N.A., Erera, A.L., Savelsbergh, M.W., \& Wang, X. (2012). Optimization for dynamic ride-sharing: A review. European Journal of Operational Research, 223(2), 295-303.

\bibitem[Ali and Dyo, 2017]{AD2017} Ali, J., \& Dyo, V. (2017). Coverage and mobile sensor placement for vehicles on predetermined routes: a greedy heuristic approach. In Proceedings of the 14th International Joint Conference on e-Business and Telecommunications (ICETE 2017) - WINSYS (pp. 83-88). SciTePress. 

\bibitem[Alonso-Mora et al., 2017]{ASWFR2017} Alonso-Mora, J., Samaranayake, S., Wallar, A., Frazzoli, E., \& Rus, D. (2017). On-demand high-capacity ride-sharing via dynamic trip-vehicle assignment. Proceedings of the National Academy of Sciences, 114(3), 462-467. 

\bibitem[Asprone et al., 2021]{ADFS2021} Asprone, D., Di Martino, S., Festa, P., \& Starace, L. L. L. (2021). Vehicular crowd-sensing: a parametric routing algorithm to increase spatio-temporal road network coverage. International Journal of Geographical Information Science, 35(9), 1876-1904.

\bibitem[Biswas et al., 2018]{BGTMMT} Biswas, A., Gopalakrishnan, R., Tulabandhula, T., Metrewar, A., Mukherjee, K., \& Thangaraj, R. S. (2018). Impact of detour-aware policies on maximizing profit in ridesharing (No. 18–00891). Transportation Research Board 97th Annual MeetingTransportation Research Board. 

\bibitem[Bock et al., 2017]{BAD2017} Bock, F., Attanasio, Y., \& Di Martino, S. (2017). Spatio-temporal road coverage of probe vehicles: a case study on crowd-sensing of parking availability with taxis. In Societal Geo-innovation: Selected papers of the 20th AGILE conference on Geographic Information Science, pp. 165-184. Springer International Publishing.

\bibitem[Bonola et al., 2016]{BBLARB2016} Bonola, M., Bracciale, L., Loreti, P., Amici, R., Rabuffi, A., \& Bianchi, G. (2016). Opportunistic communication in smart city: Experimental insight with small-scale taxi fleets as data carriers. Ad Hoc Networks, 43, 43-55.

\bibitem[Chen et al., 2020]{C2020} Chen, X., Xu, S., Han, J., Fu, H., Pi, X., Joe-Wong, C., Li, Y., Zhang, L., Noh, H.Y., \& Zhang, P. (2020). Pas: Prediction-based actuation system for city-scale ridesharing vehicular mobile crowdsensing. IEEE Internet of Things Journal, 7(5), 3719-3734.

\bibitem[Clarke, 1971]{Clarke1971} Clarke, E. H. (1971), Multipart pricing of public goods. Public Choice 11(1): 17-33.

\bibitem[Cruz et al., 2020a]{Cruz2020a} Cruz, P., Couto, R.D., Costa, L.H., Fladenmuller, A., \& Amorim, M.D. (2020). A delay-aware coverage metric for bus-based sensor networks. Computer Communications, 156, 192-200.

\bibitem[Cruz et al., 2020b]{Cruz2020b} Cruz, P., Couto, R.D., Costa, L.H., Fladenmuller, A., \& de Amorim, M. (2020). Per-vehicle coverage in a bus-based general-purpose sensor network. IEEE Wireless Communications Letters, 9, 1019-1022.

\bibitem[Furuhata et al., 2015]{FDKODBCW2015} Furuhata, M., Daniel, K., Koenig, S., Ordonez, F., Dessouky, M., Brunet, M.-E., Cohen, L., \& Wang, X. (2015). Online cost-sharing mechanism design for demand-responsive transport systems. IEEE Transactions on Intelligent Transportation Systems, 16(2), 692-707. 

\bibitem[Furuhata et al., 2013]{FDOBWK2013} Furuhata, M., Dessouky, M., Ordóñez, F., Brunet, M. E., Wang, X., \& Koenig, S. (2013). Ridesharing: The state-of-the-art and future directions. Transportation Research Part B: Methodological, 57, 28-46.

\bibitem[Fan et al., 2021]{FJLQG2021} Fan, G., Jin, H., Liu, Q., Qin, W., Gan, X., Long, H., Fu, L., \& Wang, X. (2021). Joint scheduling and incentive mechanism for spatio-temporal vehicular crowd sensing. IEEE Transactions on Mobile Computing, 20, 1449-1464.

\bibitem[Gao et al., 2016]{Gao2016} Gao, Y., Dong, W., Guo, K., Liu, X., Chen, Y., Liu, X., Bu, J., \& Chen, C. (2016). Mosaic: A low-cost mobile sensing system for urban air quality monitoring. IEEE INFOCOM 2016 - The 35th Annual IEEE International Conference on Computer Communications (pp: 1-9).

\bibitem[Groves, 1973]{Groves1973} Groves, T. (1973). Incentives in teams. Econometrica, 41(4), 617-631.

\bibitem[Guo et al., 2022]{GQDRJ2022} Guo, S., Qian, X., Dasgupta, S., Rahman, M., \& Jones, S. (2022). Sensing and monitoring of urban roadway traffic state with large-scale ride-sourcing vehicles. The Rise of Smart Cities, 551-582.

\bibitem[Guo and Qian, 2022]{GQ2022} Guo, S., \& Qian, X. (2022). Optimal sensing of urban road networks with large-scale ridesourcing vehicles. Arxiv preprint arxiv: 2207.11285.

\bibitem[Gopalakrishnan et al., 2016]{GMT2016} Gopalakrishnan, R., Mukherjee, K., \& Tulabandhula, T. (2016). The costs and benefits of sharing: Sequential individual rationality and sequential fairness. Arxiv preprint arxiv: 1607.07306.

\bibitem[Han et al., 2023]{HJNLL} Han, K., Ji, W., Nie, Y. (Marco), Li, Z., \& Liu, S. (2023). Exploring the sensing power of mixed vehicle fleets. Arxiv preprint arxiv: 2311.15237.

\bibitem[Ji et al., 2023]{JHL2023} Ji, W., Han, K., \& Liu, T. (2023). A survey of urban drive-by sensing: An optimization perspective. Sustainable Cities and Society, 104874.

\bibitem[Lee and Savelsbergh, 2015]{LS2015} Lee, A., \& Savelsbergh, M. (2015). Dynamic ridesharing: Is there a role for dedicated drivers? Transportation Research Part B: Methodological, 81, 483-497.

\bibitem[Li et al., 2020]{LNL2020} Li, R., Nie, Y. (Marco), \& Liu, X. (2020). Pricing carpool rides based on schedule displacement. Transportation Science, 54(4), 1134-1152. 

\bibitem[Li et al., 2022]{LNL2022} Li, R., Nie, Y., \& Liu, X. (2022). Auction-based permit allocation and sharing system (A-PASS) for travel demand management. Transportation Science, 56(2), 322-337.

\bibitem[Ma et al., 2022]{MFP2022} Ma, H., Fang, F., \& Parkes, D.C. (2022). Spatio-temporal pricing for ridesharing platforms. Operations Research, 70(2): 1025-1041.

\bibitem[Martino and Lucio Starace, 2022]{ML2022} Martino, S.D., \& Lucio Starace, L.L. (2022). Vehicular crowd-sensing on complex urban road networks: A case study in the city of Porto. Transportation Research Procedia, 62, 350-357.

\bibitem[Masoud and Jayakrishnan, 2017]{MJ2017} Masoud, N., \& Jayakrishnan, R. (2017). A decomposition algorithm to solve the multi-hop Peer-to-Peer ride-matching problem. Transportation Research Part B: Methodological, 99, 1-29. 

\bibitem[Masutani, 2015]{Masutani2015} Masutani, O. (2015). A sensing coverage analysis of a route control method for vehicular crowd sensing. In 2015 IEEE International Conference on Pervasive Computing and Communication Workshops (PerCom Workshops) (pp. 396-401). 

\bibitem[Mathur et al., 2010]{MJKCXGT2010} Mathur, S., Jin, T., Kasturirangan, N., Chandrasekaran, J., Xue, W., Gruteser, M., \& Trappe, W. (2010). Parknet: drive-by sensing of road-side parking statistics. In Proceedings of the 8th international conference on Mobile systems, applications, and services (pp. 123-136).

\bibitem[Myerson and Satterthwaite, 1983]{MS1983} Myerson, R. B., \& Satterthwaite, M. A. (1983). Efficient mechanisms for bilateral trading. Journal of economic theory, 29(2), 265-281.

\bibitem[O'Keeffe et al., 2019]{OASSR2019} O’Keeffe, K. P., Anjomshoaa, A., Strogatz, S. H., Santi, P., \& Ratti, C. (2019). Quantifying the sensing power of vehicle fleets. Proceedings of the National Academy of Sciences, 116(26), 12752-12757.

\bibitem[Pasqualetti et al., 2012]{PDB2012} Pasqualetti, F., Durham, J. W., \& Bullo, F. (2012). Cooperative patrolling via weighted tours: Performance analysis and distributed algorithms. IEEE Transactions on Robotics, 28(5), 1181-1188.

\bibitem[Pelzer et al., 2015]{PXZLKA2015} Pelzer, D., Xiao, J., Zehe, D., Lees, M. H., Knoll, A. C., \& Aydt, H. (2015). A partition-based match making algorithm for dynamic ridesharing. IEEE Transactions on Intelligent Transportation Systems, 16(5), 2587-2598.

\bibitem[Qian and Ukkusuri, 2017]{QU2017} Qian, X., \& Ukkusuri, S. V. (2017). Taxi market equilibrium with third-party hailing service. Transportation Research Part B: Methodological, 100, 43-63.

\bibitem[Song et al., 2021]{SHS2021} Song, J., Han, K., \& Stettler, M. E. (2020). Deep-MAPS: Machine learning based mobile air pollution sensing. IEEE Internet of Things Journal, 8(9), 7649-7660.

\bibitem[Stiglic et al., 2015]{SASG2015} Stiglic, M., Agatz, N., Savelsbergh, M., \& Gradisar, M. (2015). The benefits of meeting points in ride-sharing systems. Transportation Research Part B: Methodological, 82, 36-53.

\bibitem[Stiglic et al., 2016]{SASG2016} Stiglic, M., Agatz, N., Savelsbergh, M., \& Gradisar, M. (2016). Making dynamic ride-sharing work: The impact of driver and rider flexibility. Transportation Research Part E: Logistics and Transportation Review, 91, 190-207.

\bibitem[Tonekaboni et al., 2020]{TRMSO2020} Tonekaboni, N.H., Ramaswamy, L., Mishra, D., Setayeshfar, O., \& Omidvar, S. (2020). Spatio-temporal coverage enhancement in drive-by sensing through utility-aware mobile agent selection. In: Song, W., Lee, K., Yan, Z., Zhang, LJ., Chen, H. (eds) Internet of Things - ICIOT 2020. ICIOT 2020. Lecture Notes in Computer Science, vol 12405, pp. 108-124. Springer, Cham: Springer International Publishing.

\bibitem[Vickrey, 1961]{Vickrey1961} Vickrey, W. (1961). Counterspeculation, auctions, and competitive  sealed tenders. The Journal of Finance, 16(1): 8-37.

\bibitem[Wang et al., 2018]{WAE2018} Wang, X., Agatz, N., \& Erera, A. (2018). Stable matching for dynamic ride-sharing systems. Transportation Science, 52(4), 850-867.

\bibitem[Xu et al., 2019]{XCPJZN2019} Xu, S., Chen, X., Pi, X., Joe-Wong, C., Zhang, P., \& Noh, H.Y. (2020). iLOCuS: Incentivizing vehicle mobility to optimize sensing distribution in crowd sensing. IEEE Transactions on Mobile Computing, 19, 1831-1847.

\bibitem[Yang et al., 2002]{YWW2002} Yang, H., Wong, S. C., \& Wong, K. I. (2002). Demand-supply equilibrium of taxi services in a network under competition and regulation. Transportation Research Part B: Methodological, 36(9), 799-819. 

\bibitem[Yu et al., 2021]{YFLRL2021} Yu, H. ,  Fang, J. ,  Liu, S. ,  Ren, Y. , \&  Lu, J . (2021). A node optimization model based on the spatiotemporal characteristics of the road network for urban traffic mobile crowd sensing. Vehicular Communications, 31, 100383.

\bibitem[Zhang et al., 2018]{ZWB2018} Zhang, C., Wu, F., \& Bei, X. (2018). An efficient auction with variable reserve prices for ridesourcing. In PRICAI 2018: Trends in Artificial Intelligence: 15th Pacific Rim International Conference on Artificial Intelligence, Nanjing, China, August 28–31, 2018, Proceedings, Part I 15, pp. 361-374. Springer International Publishing.

\bibitem[Zhang et al., 2016]{ZWZ2016} Zhang, J., Wen, D., \& Zeng, S. (2016). A discounted trade reduction mechanism for dynamic ridesharing Pricing. IEEE Transactions on Intelligent Transportation Systems, 17(6), 1586-1595. 

\bibitem[Zhao et al., 2014]{ZZGSY2014} Zhao, D., Zhang, D., Gerding, E., Sakurai, Y., \& Makoto, Y. (2014). Incentives in ridesharing with deficit control. In 13th International Conference on Autonomous Agents and Multiagent Systems, AAMAS 2014. International Foundation for Autonomous Agents and Multiagent Systems (IFAAMAS), 2014. 

\bibitem[Zhao et al., 2015]{ZMLL2015} Zhao, D., Ma, H., Liu, L., \& Li, X. (2015). Opportunistic coverage for urban vehicular sensing. Computer Communications, 60, 71-85.
 
\end{thebibliography}
\end{document}